\documentclass[12pt]{article} 

\usepackage{amsmath}
\usepackage{amsthm}
\usepackage{amsfonts}
\usepackage{mathrsfs}
\usepackage{stmaryrd}
\usepackage{setspace}
\usepackage{fullpage}
\usepackage{amssymb}
\usepackage{breqn}
\usepackage{enumitem}
\usepackage{bbold} 
\usepackage{authblk}
\usepackage{comment}
\usepackage{hyperref}
\usepackage{pgf,tikz}
\usepackage{graphicx}
\usepackage{subcaption}

\newtheorem{thm}{Theorem}[section]%[chapter]

\newtheorem{proposition}[thm]{Proposition}

\newtheorem{corollary}[thm]{Corollary}

\newtheorem{observation}[thm]{Observation}

\newcommand\ex{\ensuremath{\mathrm{ex}}}
\newcommand\thre{\ensuremath{\mathrm{th}}}

\newcommand\cH{{\mathcal H}}

\newcommand\cN{{\mathcal N}}

\newcommand{\ignore}[1]{}

\title{A note on the uniformity threshold for Berge hypergraphs}
\author{Dániel Gerbner} 

\date{\small Alfr\'ed R\'enyi Institute of Mathematics}

\begin{document}

\maketitle

\begin{abstract}
    A Berge copy of a graph is a hypergraph obtained by enlarging the edges arbitrarily.
    Gr\'osz, Methuku and Tompkins in 2020 showed that for any graph $F$, there is an integer $r_0=r_0(F)$, such that for any $r\ge r_0$, any $r$-uniform hypergraph without a Berge copy of $F$ has $o(n^2)$ hyperedges. The smallest such $r_0$ is called the uniformity threshold of $F$ and is denoted by $\ensuremath{\mathrm{th}}(F)$. They showed that $\ensuremath{\mathrm{th}}(F)\le R(F,F')$, where $R$ denotes the off-diagonal Ramsey number and $F'$ is any graph obtained form $F$ by deleting an edge.
    
    We improve this bound to $\ensuremath{\mathrm{th}}(F)\le R(K_{\chi(F)},F')$, and use the new bound to determine $\ensuremath{\mathrm{th}}(F)$ exactly for several classes of graphs.
\end{abstract}

\section{Introduction}

Given a graph $G$ and a hypergraph $\cH$, we say that $\cH$ is a \textit{Berge} copy of $G$ (Berge-$G$ in short) if $V(G)\subset V(\cH)$ and there is a bijection $f:E(G)\rightarrow E(\cH)$ such that for any edge $e\in E(G)$, we have $e\subset f(e)$. We also say that $G$ is the \textit{core} of $\cH$. Note that there are several non-isomorphic Berge copies of a graph, and a hypergraph is a Berge copy of several non-isomorphic graphs. Also, a hypergraph can have multiple isomorphic cores.

Berge copies (or Berge hypergraphs), extending the notion of hypergraph cycles due to Berge, were introduced by Gerbner and Palmer \cite{gp1}. Since then, extremal problems for Berge hypergraphs have attracted a lot of attention, see Section 5.2.2 of \cite{gp} for a survey.

In this paper we are concerned with the maximum number of hyperedges in $r$-uniform $n$-vertex hypergraphs that do not contain any Berge copy of a given graph $F$ (in short: Berge-$F$-free hypergraphs). We denote this quantity by $\ex_r(n,\textup{Berge-}F)$. Note that a 2-uniform Berge copy of $F$ is $F$, and $\ex(n,F):=\ex_2(n,\textup{Berge-}F)$ is the Tur\'an number of $F$.

Gerbner and Palmer showed that if the uniformity is large enough, then a Berge-$F$-free hypergraph can have at most quadratic many hyperedges.

\begin{proposition}[Gerbner, Palmer,\cite{gp1}]
If $r\ge |V(F)|$, then $\ex_r(n,\textup{Berge-}F)=O(\ex(n,F))=O(n^2)$.
\end{proposition}

This was improved by Gr\'osz, Methuku and Tompkins \cite{GMT} in the case the uniformity is even larger.

\begin{thm}[Gr\'osz, Methuku, Tompkins \cite{GMT}]
For any graph $F$, there is an integer $r_0=r_0(F)$ such that for any $r\ge r_0$, $\ex_r(n,\textup{Berge-}F)=o(n^2)$.
\end{thm}

The smallest possible $r_0$ in the above theorem is called the \textit{uniformity threshold} of $F$ and is denoted by $\thre(F)$. 
Gr\'osz, Methuku and Tompkins \cite{GMT} initiated the study of the uniformity threshold. They proved the following general upper bound.

\begin{thm}[Gr\'osz, Methuku, Tompkins \cite{GMT}]\label{grmeto}
For any graph $F$ and any edge $e$ of $F$, we have $\thre(F)\le R(F,F\setminus e)$.
\end{thm}

Here we let $F\setminus e$ denote an arbitrary graph obtained by deleting an edge from $F$.
$R$ denotes the off-diagonal Ramsey number, i.e. $R(H,G)$ is the smallest number $r$ of vertices such that if we color each the edge of $K_r$ to blue or red, then we can find either a mono-blue $H$ or a mono-red $G$.

Using the above theorem, Gr\'osz, Methuku and Tompkins \cite{GMT} determined the uniformity threshold exactly for seven graphs on at most 5 vertices, including the triangle. Other than those, the uniformity threshold is known only in some "easy" cases, where $\ex_r(n,\textup{Berge-}F)$ is well studied for every $r$. For example, $\thre(F)=2$ for bipartite graphs with a vertex whose removal results in a forest \cite{gmv}, and $\thre(F)=3$ for odd cycles of length more than three \cite{gyl}.

Our main result is the following bound.

\begin{thm}\label{main}
For any graph $F$ and any edge $e$ of $F$, we have $\thre(F)\le R(K_{\chi(F)},F\setminus e)$.
\end{thm}

We prove another bound for a class of graphs. Before stating our next theorem, we have to introduce a notion closely connected to Berge hypergraphs.

Given graphs $H$ and $G$, we let $\cN(H,G)$ denote the number of copies of $H$ in $G$. Given $F$, we let $\ex(n,H,F):=\{\cN(H,G): G \text{ is an $F$-free $n$-vertex graph}\}$. This quantity is called the \textit{generalized Tur\'an number} of $H$ and $F$. After several sporadic result, the systematic study of generalized Tur\'an problems was initiated by Alon and Shikhelman \cite{as}.

The connection to Berge hypergraphs was established by Gerbner and Palmer \cite{gp2}, who showed that $\ex(n,K_r,F)\le \ex_r(n,\textup{Berge-}F)\le\ex(n,K_r,F)+\ex(n,F)$.

\begin{thm}\label{kicsi}
If $\ex(n,K_3,F)=o(n^2)$, then $\thre(F)\le (\chi(F)-1)(|V(F)|-1)+1$.
\end{thm}

Using the above upper bounds and known results on Ramsey numbers, we can determine the uniformity threshold exactly for several classes of graphs. Let $F_k$ denote the \textit{$k$-fan}, $k$ triangles sharing a vertex. Let $B_t$ denote the \textit{book with $t$ pages}, $t$ triangles sharing an edge. Let $W_k$ denote the \textit{wheel with $k$ spokes}, a $C_k$ with an additional vertex connected to each vertex of the cycle. Let $B_{p,q}(m)$ denote the \textit{generalized book}, $m$ copies of $K_q$ each sharing a fixed set of $p$ vertices.

\begin{thm}\label{rams}
We have  $\thre(F)= (\chi(F)-1)(|V(F)|-1)+1$ if $F$ is the $k$-fan with $k>1$, $B_t$ with $t>1$, wheel graph $W_k$ with $k\ge 6$ an even number, or a generalized book $B_{p,q}(m)$ with $m$ large enough. In particular

\begin{itemize}
    \item $\thre(F_k)=4k+1$ for $k>1$,
    \item $\thre(B_t)=2t+3$ for $t>1$,
    \item $\thre(W_k)=2k+1$ for even $k>4$,
    \item $\thre(B_{p,q}(m))=(q-1)(m(q-p)+p-1)+1$ for $m$ large enough.
\end{itemize}
\end{thm}

The structure of the paper is as follows. In Section 2, we prove our theorems through a series of propositions. 
%introduce a new notion that we will use later. Section 3 contains the proof of Theorems \ref{}, and 
We finish the paper with some concluding remarks in Section 3.

\section{Proofs}

We will use the removal lemma \cite{efr}. It states that if an $n$-vertex graph $G$ has $o(n^{|V(H)|})$ copies of $H$, then we can make $G$ $H$-free by removing $o(n^2)$ edges.

The \textit{shadow graph} of a hypergraph $\cH$ is the graph $G$ on the same vertex set with $uv\in E(G)$ if and only if there is a hyperedge of $\cH$ containing both $u$ and $v$.

Lu and Wang \cite{lw} initiated the study of the maximum number of edges in the shadow graph of a Berge-$F$-free $r$-uniform $n$-vertex hypergraph. This quantity is called the \textit{$r$-cover Tur\'an number} of $F$ and is denoted by $\hat{\ex}_r(n,\textup{Berge-}F)$. We initiate the study of a generalized version of this, the straightforward analogy of generalized Tur\'an numbers. We let $\hat{\ex}_r(n,H,\textup{Berge-}F)$ denote the maximum number of copies of $H$ in the shadow graph of an $r$-uniform Berge-$F$-free $n$-vertex hypergraph.

\begin{proposition}\label{propp} If $r\ge |V(F)|$, then $\hat{\ex}_r(n,H,\textup{Berge-}F)\le O(\ex(n,H,F))+O(n^{|V(H)|-1})$.
%$\hat{\ex}_r(n,H,\textup{Berge-}F)=o(n^{|V(H)|}$ if and only if $\ex(n,H,F)=o(n^{|V(H)|}$. ez majd coroll
\end{proposition}

\begin{proof} 
Let $\cH$ be a Berge-$F$-free $r$-uniform $n$-vertex hypergraph. Recall that $\cH$ has $O(n^2)$ hyperedges by a result of Gerbner and Palmer \cite{gp1}.
We distinguish two types of copies of $H$ in the shadow graph $G$ of $\cH$. Those copies of $H$ that contain at least 3 vertices from some hyperedge of $\cH$ can be counted by picking a hyperedge $O(n^2)$ ways, picking 3 vertices of it constant many ways, and then picking $V(H)-3$ other vertices, $O(n^{|V(H)|-3})$ many ways. Therefore, there are $O(n^{|V(H)|-1})$ such hyperedges.

Let us consider now the other copies of $H$. For each hyperedge of $\cH$, let us pick a sub-edge randomly with uniform distribution, independently from the other edges. Let $G'$ be the graph having those edges. Then $G'$ is clearly $F$-free. Observe that a copy of $H$ that shares at most two vertices with any hyperedge of $\cH$ is in $G'$ with probability at least $1/\binom{r}{2}^{|E(H)|}$. Indeed, every edge of the copy of $H$ is in at least one hyperedge of $\cH$, thus it is in $G'$ with probability at least $1/\binom{r}{2}$.
Distinct edges of the copy of $H$ are in $G'$ independently of each other, as they may be included in $G'$ only via distinct hyperedges. This implies that the number of the copies of the second type of $H$ is at most $\binom{r}{2}^{|E(H)|}\cN(H,G')\le \binom{r}{2}^{|E(H)|}\ex(n,H,F)$, completing the proof.
\end{proof}

We are going to be interested in the case when $\hat{\ex}_r(n,H,\textup{Berge-}F)=o(n^{|V(H)|})$. By the above result, it holds if
$\ex(n,H,F)=o(n^{|V(H)|})$. By a result of Alon and Shikhelman \cite{as} this happens if and only if $F$ is a subgraph of a blow-up of $H$.

\begin{corollary}\label{coro1} If $r\ge |V(F)|$ and $F$ is a subgraph of a blow-up of $H$, then
$\hat{\ex}_r(n,H,\textup{Berge-}F)=o(n^{|V(H)|})$.
\end{corollary}

What we are actually interested in is the following quantity. We say that a set $E'$ of edges cover a subgraph of $G$ if the subgraph contains an edge from $G$. Let $C(H,G)$ denote the minimum number of edges in $G$ that cover each copy of $H$. Let $x(n,H,F)$ denote the largest value of $C(H,G)$ in $F$-free $n$-vertex graphs $G$. Let $\hat{x}_r(n,H,\textup{Berge-}F)$ denote the largest value of $C(H,G)$ in the shadow graph $G$ of a Berge-$F$-free $n$-vertex $r$-uniform hypergraph. Corollary \ref{coro1} and the removal lemma imply the following.

\begin{corollary}\label{coro2} If $r\ge |V(F)|$ and $F$ is a subgraph of a blow-up of $H$, then
$\hat{x}_r(n,H,\textup{Berge-}F)=o(n^{2})$. In particular, $\hat{x}_r(n,K_{\chi(F)},\textup{Berge-}F)=o(n^{2})$.
\end{corollary}

%\section{Proof of Theorems \ref{main} and \ref{kicsi}}

Given a hypergraph $\cH$, we say that an edge $uv$ of the shadow graph is $t$-heavy if $u$ and $v$ are contained in at least $t$ hyperedges of $\cH$. Otherwise $uv$ is $t$-light. We say that a subgraph of the shadow graph is $t$-heavy (resp. $t$-light) if each edge of it is $t$-heavy (resp. $t$-light).

Let us describe first the main advantage of heavy edges. 

\begin{observation}\label{obs1}
Assume that $t\ge |E(F)|$ %or $t\ge \binom{r}{2}$ 
and we find a Berge copy of a subgraph $F'$ of $F$ in $\cH$, such that its core is extended to a copy of $F$ with $t$-heavy edges in the shadow graph $G$. Then this copy is the core of a Berge-$F$ in $\cH$.
\end{observation} 

\begin{proof}
All we need to do is to pick distinct hyperedges for the additional edges. We go through those additional edges in an arbitrary order, and pick such an edge arbitrarily. There are at most $|E(F)|-1$ hyperedges picked earlier (either already in the original Berge copy of $F'$, or picked for an earlier one of the additional edges). Thus we can pick a new hyperedge for each of the additional edges, to complete the Berge copy of $F$.
\end{proof} 

In particular, this implies that if $\cH$ is Berge-$F$-free, then there is no $t$-heavy copy of $F$ in the shadow graph.

\begin{proposition}\label{lig}
If $\cH$ is a Berge-$F$-free $r$-uniform hypergraph, then there are at most $(t-1)\hat{x}_r(n,H,\textup{Berge-}F)$ hyperedges in $\cH$ containing a $t$-light copy of $H$.
%$F$ is not a blow-up of $H$ and $t$ is an integer that does not depend on $n$, then there are $o(n^2)$ hyperedges in $\cH$ containing a $t$-light copy of $H$.
\end{proposition}

\begin{proof}
Let $G$ be the shadow graph of $\cH$ and $S$ be a set of $\hat{x}_r(n,H,\textup{Berge-}F)$ edges covering each copy of $H$. Then there is a $t$-light edge of $S$ inside every hyperedge that contains a $t$-light copy of $H$. Each $t$-light edge of $S$ is counted in at most $t-1$ hyperedges, thus there are at most $(t-1)|S|$ hyperedges, completing the proof.
%By Corollary \ref{coro2}, there is a set $S$ of $o(n^2)$ edges that covers every copy of $H$ in $G$. In particular, there is a $t$-light edge of $S$ inside every hyperedge of $\cH$. Then there are at most $(t-1)|S|$ hyperedges, finishing the proof.
\end{proof}

\begin{proposition}\label{fo}
If $F$ is a subgraph of a blow-up of $H$, then $\thre(F)\le R(H,F\setminus e)$.
\end{proposition}

\begin{proof} Let $\cH$ be an $r$-uniform Berge-$F$-free $n$-vertex hypergraph and let $t=|E(F)|$. We apply Proposition \ref{lig} to obtain that there are at most $(t-1)\hat{x}_r(n,H,\textup{Berge-}F)$ hyperedges in $\cH$ containing a $t$-light copy of $H$. By Corollary \ref{coro2}, $\hat{x}_r(n,H,\textup{Berge-}F)=o(n^2)$, thus there are $o(n^2)$ hyperedges containing a $t$-light copy of $H$. Let $\cH'$ denote the subhypergraph obtained by deleting these $o(n^2)$ hyperedges. We will show that if $r\ge R(H,F\setminus e)$, then $\cH'$ is empty.
%also has $o(n^2)$ hyperedges.

By Observation \ref{obs1}, if a hyperedge $h_1$ of $\cH'$ contains a $t$-heavy copy of $F\setminus e$, then that copy is the core of a Berge-$F\setminus e$ in $\cH'\setminus h$. 
%Indeed, we go through the edges of $F\setminus e$ in an arbitrary order. Then for each edge we pick a hyperedge containing it that has not been picked earlier. This is doable, as there are at least $t-1=|E(F)|-1$ hyperedges containing a given edge, and at most $|E(F)|-2$ of them were picked earlier. 
Now we can add $h$, representing $e$, to obtain a copy of Berge-$F$ in $\cH'$, a contradiction. We obtained that every hyperedge $h$ of $\cH'$ contains no $t$-heavy $F\setminus e$, nor $t$-light $H$. But $t$-heavy and $t$-light gives a 2-coloring of $h$, thus $h$ has less than $R(H,F\setminus e)\le r$ vertices, a contradiction finishing the proof.
\end{proof}

Theorem \ref{main} immediately follows from the above result by observing that $F$ is a subgraph of a blow-up of $K_{\chi(F)}$. We are also ready to prove Theorem \ref{kicsi}. Recall that it states the upper bound $\thre(F)\le (\chi(F)-1)(|V(F)|-1)+1$ if $\ex(n,K_3,F)=o(n^2)$.

\begin{proof}[Proof of Theorem \ref{kicsi}]
Let $\cH$ be a Berge-$F$-free $r$-uniform $n$-vertex hypergraph and $t=|E(F)|$. As in the proof of Proposition \ref{fo}, we have that no hyperedge contains a $t$-heavy $F\setminus e$ and $o(n^2)$ hyperedges contain a $t$-light copy of $K_{\chi(F)}$. Let $\cH_1$ denote the subhypergraph of those hyperedges that contain a triangle with two $t$-heavy edges and one $t$-light edge. We claim that $\cH_1$ has $o(n^2)$ hyperedges. 

Let us pick a triangle with two $t$-heavy edges and one $t$-light edge for each hyperedge in $\cH_1$ and let $G_1$ denote the graph that contains the three edges picked for each hyperedge. We claim that $G_1$ is $F$-free. Indeed, assume that there is a copy of $F$. First we pick the hyperedges for the $t$-light edges in that copy. This is doable, as for each of those edges $uv$ there is at least one hyperedge $h$ that contains $uv$ such that $uv$ is in the triangle we picked for $h$. Then we pick $h$ for $uv$. This way we pick every hyperedge $h$ at most once, as we picked only one $t$-light edge for $h$. Now we have a Berge copy of a subgraph of $F$, and the core of this subgraph is extended to $F$ by $t$-heavy edges, thus we can use Observation \ref{obs1} to find a Berge-$F$ in $\cH_1$, a contradiction.

Thus we have that $G_1$ is $F$-free and hence by our assumption has $o(n^2)$ triangles. For every hyperedge of $\cH_1$, we picked a triangle in $G_1$. Observe that each triangle $uvw$ was picked at most $t-1$ times. Indeed, the triangle contains a $t$-light edge $uv$, thus at most $t-1$ hyperedges contain this triangle. This implies that $\cH_1$ has $o(n^2)$ hyperedges. 

%Let $\cH_2$ denote the subhypergraph of those hyperedges without a $t$-light copy of $K_{\chi(F)}$ and not in $\cH_1$. We will show that if $r\ge (\chi(F)-1)(|V(F)|-1)+1$, then $\cH_2$ is empty. Indeed, if $h$ is a hyperedge of $\cH_2$, 

We will show that there are no further hyperedges if $r\ge (\chi(F)-1)(|V(F)|-1)+1$. Assume that $h$ is a hyperedge without a $t$-light copy of $K_{\chi(F)}$ not in $\cH_1$, and let us consider its subedges. By forbidding triangles with exactly two $t$-heavy edges, we obtain that $t$-heavy edges form vertex-disjoint cliques in $h$. Such a clique has at most $|V(F)|-1$ vertices, as otherwise we have a $t$-heavy copy of $F$. If there are $k$ vertex-disjoint $t$-heavy cliques, then there is a $t$-light $K_k$, thus $k<\chi(F)$. Hence $h$ has at most $(\chi(F)-1)(|V(F)|-1)$ vertices, a contradiction.
\end{proof}

Let us turn our attention to lower bounds.
Gr\'osz, Methuku and Tompkins \cite{GMT} proved a general lower bound. We say that a partition of $V(F)$ into
sets of size at most $t$ is a \textit{$t$-admissible partition} of $F$ if there is at most one edge between any two sets in $F$. Given a $t$-admissible partition $P$, we let $F(P)$ be the graph which has the sets of the partition as vertices, with $UU'$ being an edge if and only if there is a vertex in $U$ connected to a vertex of $U'$.
Given $F$, we let $c_t(F)$ denote the smallest chromatic number of graphs $F(P)$, where $P$ is a $t$-admissible partition.

\begin{thm}[Gr\'osz, Methuku, Tompkins \cite{GMT}]
Let $F$ be a graph with $c_t(F)\ge 3$ and $1\le t\le |V(F)|-1$. Then $\thre(F)\ge (c_t(F)-1)t+1$.
\end{thm}

Observe that partitioning into singletons is $t$-admissible, thus $c_t(F)\le \chi(F)$. This means that the lower bound on $\thre(F)$ is at most $(\chi(F)-1)(|V(F)|-1)+1$. Let us compare this to our upper bounds.
Theorem \ref{kicsi} has the same quantity as an upper bound, and $(\chi(F)-1)(|V(F)|-1)+1$ is also a well-known and easy lower bound on $R(K_{\chi(F)},F\setminus e)$ if $F\setminus e$ is connected \cite{chv}. 
%Indeed, we can take $\chi(F)-1$ mono-red copies of $K_{|V(F)|-1}$, and color each other edge blue. Then the blue graph has chromatic number less than $\chi(F)$, while the red graph does not have a component of order $|V(F)|$, thus there is neither blue $K_{\chi(F)}$, nor red $F\setminus e$ in the resulting graph.

This means that to find graphs where we can determine the uniformity threshold, we should check the cases of equality in the lower and upper bounds.

\begin{corollary}
Let $F$ consist of a subgraph $F_0$ and an additional vertex $v$ connected to each vertex of $F_0$. If $F_0$ is connected or $F_0$ is a bipartite graph such that at least two of its components contain an edge,
%...  $F$ has a vertex $v$ connected to each other vertex and the other vertices form a connected graph or a bipartite graph such that at least two of its components contain an edge,
%If every edge of $F$ is in a triangle and $F$ is connected, 
then $\thre(F)\ge (\chi(F)-1)(|V(F)|-1)+1$.
\end{corollary} 

\begin{proof}
Let us consider a $(|V(F)|-1)$-admissible partition $P$ of $F$. If $\{v\}$ is a part, then each other vertex is a part, thus $\chi(F(P))=\chi(F)$. Otherwise $v$ is in a part with another vertex $u$. Then each other neighbor of $u$ must be in the same part. If $F\setminus v$ is connected, then all the vertices must be in that part, a contradiction. 

Otherwise, a connected component of $F_0$ is in the part of $v$. Then the other vertices $w$ must form parts $\{w\}$, as otherwise a part would contain two neighbors of $v$. As $F_0$ contains an edge not in the component of $v$, we have $\chi(F(P))\ge\chi(F)=3$.
%We will show that the only $t$-admissible partition of $F$ is the partition into singletons. Observe that the vertices of a triangle are either in one part, or in three parts. ja nem jó, szélső háromszög lehetne egyben... persze attól chrom szám nem nő... ja így $F_2$-re se
\end{proof}

Let us consider now when $R(F\chi(F),F\setminus e)=(\chi(F)-1)(|V(F)|-1)+1$.
This is in fact a well-studied notion in Ramsey theory. A graph $G$ is called \textit{$p$-good} if $R(K_p,G)=(p-1)(|V(G)|-1)+1$.

Now Theorem \ref{rams} follows from Theorem \ref{main} and known results in Ramsey theory that we list below.
Li and Rousseau \cite{lr} showed that the $k$-fan is 3-good for every $k>1$. 
Rousseau and Sheehan \cite{rs} showed that the book $B_t$ is 3-good for every $t>1$. Burr and Erd\H os \cite{be} showed that the wheel $W_k$ is 3-good for $k\ge 5$.
%(nekünk ebből csak odd jó. vagy nem, mert $F-e$ van?...).
Nikiforov and Rousseau \cite{nr} showed that the generalized book $B_{p,q}(m)$ is $q$-good for $m$ large enough. The threshold on $m$ was improved in \cite{fhw}.

We remark that Theorem \ref{kicsi} also implies two of the statements in Theorem \ref{rams} using the following known results in generalized Tur\'an theory. Alon and Shikhelman \cite{as} showed that $\ex(n,K_3,B_t)=o(n^2)$ and $\ex(n,K_3,F_k)=O(n)$.
%wheels with at least 5 spokes

\section{Concluding remarks}

The main question regarding the uniformity threshold for Berge hypergraphs is whether it can grow exponentially with the number of vertices for some graphs. Our general upper bounds rely on the Ramsey number, which can. Theorem \ref{main}, together with a theorem from \cite{aks} shows that for graphs $F$ with chromatic number $k$, $\thre(F)$ is polynomial in $|V(F)|$, where the degree of the polynomial may depend on $k$. However, this is still exponential if $\chi(F)$ grows with $|V(F)|$, e.g. for cliques. In fact, for cliques Theorem \ref{main} does not improve Theorem \ref{grmeto}. On the other hand, the lower bound on $\thre(K_k)$ is quadratic for $k$.

We initiated the study of generalized cover Tur\'an numbers for Berge hypergraphs. We only proved Proposition \ref{propp} and its corollaries concerning this notion, but we believe that the question is interesting in general. %learning more on the number of copies of some subgraphs in the shadow graph could prove to be useful in other problems as well. 
Let us mention that $\hat{\ex}_r(n,F,\textup{Berge-}F)$ may be interesting as well.

Instead of Theorem \ref{kicsi}, we proved the way more general Proposition \ref{fo}. However, we could not show any example where applying Proposition \ref{fo} could improve the bound on $\thre(F)$.

\bigskip

\textbf{Funding}: Research supported by the National Research, Development and Innovation Office - NKFIH under the grants KH 130371, SNN 129364, FK 132060, and KKP-133819.

\end{document}